\documentclass[12pt]{article}

\usepackage{amsmath, amsthm, amsfonts, amssymb}
\newtheorem{theorem}{Theorem}
\newtheorem{definition}{Definition}

\newtheorem{remark}{Remark}
\newtheorem{corollary}{Corollary}
\newtheorem{proposition}{Proposition}
\newtheorem{example}{Example}
\usepackage[left=1.5cm,right=1.5cm,top=1.5cm,bottom=1.5cm,includeheadfoot]{geometry}

\usepackage{graphicx}
\usepackage{stmaryrd}
\usepackage{arydshln}
\usepackage[dvipsnames]{xcolor}
\usepackage{tikz}
\usepackage{latexsym, amsmath}
\usetikzlibrary{shapes,snakes}
\usetikzlibrary{matrix}
\usetikzlibrary{intersections}
\usetikzlibrary{calc}
\usetikzlibrary{positioning}
\usetikzlibrary{arrows}
\usetikzlibrary{fit}
\usepackage{textpos}
\usepackage{mathdots}
\usepackage{enumitem}
\setitemize{leftmargin=*}
\usepackage[autostyle]{csquotes}
\newcommand{\CC}{\mathbb{C}}
\newcommand{\MM}{\textnormal{M}}

\title{Structure-preserving diagonalization of matrices in \\ indefinite inner product spaces}
\author{Philip Saltenberger\\
  \small Institut Computational Mathematics, AG Numerik, TU Braunschweig\\
  \small Universit{\"a}tsplatz 2, 38106 Braunschweig\\
  \small Germany
}

\begin{document} 
\maketitle

\abstract{In this work some results on the structure-preserving diagonalization of selfadjoint and skewadjoint matrices in indefinite inner product spaces are presented. In particular, necessary and sufficient conditions on the symplectic diagonalizability of (skew)-Hamiltonian matrices and the perplectic diagonalizability of per(skew)-Hermitian matrices are provided. Assuming the structured matrix at hand is additionally normal, it is shown that any symplectic or perplectic diagonalization can always be constructed to be unitary. As a consequence of this fact, the existence of a unitary, structure-preserving diagonalization is equivalent to the existence of a specially structured additive decomposition of such matrices. The implications of this decomposition are illustrated by several examples.}

\section{Introduction} \label{intro-sec} \normalsize
Structured matrices are omnipresent in many areas of mathematics. For instance,
in the theory of matrix equations \cite{LaRod95}, structures arising from the consideration of selfadjoint and skewadjoint matrices with respect to certain inner products play a crucial role . Often, these inner products are indefinite, so that the underlying bilinear or sesquilinear form does not define a scalar product. Hence, results from Hilbert-space-theory are not available in this case and an independent mathematical analysis is required. In this work, some results in this direction are presented.

Considering the (definite) standard Euclidean inner product $\langle x,y \rangle = x^Hy, x,y \in \mathbb{C}^m$, on $\mathbb{C}^m \times \mathbb{C}^m$, it is well-known that selfadjoint and skewadjoint matrices (i.e. Hermitian and skew-Hermitian matrices) have very special properties. For example, (skew)-Hermitian matrices are always diagonalizable by a unitary matrix. The unitary matrices constitute the automorphism group of the scalar product $\langle \cdot, \cdot \rangle$ which means that $\langle Gx, Gy \rangle = \langle x,y \rangle$ always holds for any unitary matrix $G$ and all vectors $x,y \in \mathbb{C}^m$. The automorphism group is sometimes called the \textit{Lie-group} with respect to $\langle \cdot, \cdot \rangle$ whereas the selfadjoint and skewadjoint matrices are referred to as the Jordan and Lie algebras \cite{Mack07}. The Euclidean scalar product is a special case of a sesquilinear form $[x,y] = x^HBy$ on $\mathbb{C}^m \times \mathbb{C}^m$ with $B = I_m$ being the $m \times m$ identity matrix. Often, sesquilinear forms $[x,y]=x^HBy$ appear in mathematics where $B \neq I_m$. In particular, cases that have been intensively studied are those where $B$ is some (positive/negative definite or indefinite) Hermitian matrix \cite{GoLaRod05} or a skew-Hermitian matrix \cite{BuGeFass15}. The Lie-group, the Lie algebra and the Jordan algebra are defined analogously to the Euclidean scalar product for such forms as the group of automorphisms, selfadjoint and skewadjoint matrices with respect to $[x,y]=x^HBy$.

In this work, selfadjoint and skewadjoint matrices with respect to indefinite Hermitian or skew-Hermitian sesquilinear forms are considered from the viewpoint of diagonalizability. In particular, since Hermitian and skew-Hermitian matrices are always diagonalizable by a unitary matrix (i.e. an automorphism with respect to $\langle \cdot, \cdot \rangle$), we will consider the question under what conditions a similar statement holds for the automorphic diagonalization of selfadjoint and skewadjoint matrices with respect to other (indefinite) sesquilinear forms. For two particular sesquilinear forms (the symplectic and the perplectic sesquilinear form) this question will be fully analyzed and answered in Sections \ref{sec:sesf} and \ref{sec:diag}. For the symplectic bilinear form, this question was already addressed in \cite{CruzFass16}. In Section \ref{sec:normal} we consider these results in the context of normal matrices for which there always exists a unitary diagonalization. In particular, the results presented in this section apply to selfadjoint and skewadjoint matrices for which a unitary diagonalization exists. We will show that this subclass of matrices has very nice properties with respect to unitary and automorphic diagonalization and how both types of diagonalizations interact. In Section \ref{sec:notation} the notation used throughout this work is introduced whereas in Section \ref{sec:conclusions} some concluding remarks are given.

\section{Notation}
\label{sec:notation}

For any $m \in \mathbb{N}$ and $\mathbb{K} = \mathbb{R}, \CC$ we denote by $\mathbb{K}^m$ the $m$-dimensional vector space over $\mathbb{K}$ and by  $\MM_{m \times m}(\mathbb{K})$ the vector space of all $m \times m$ matrices over $\mathbb{K}$. The vector subspace $\mathcal{X}$ of $\mathbb{K}^m$ which is obtained from all possible linear combinations of some vectors $x_1, \ldots , x_k \in \mathbb{K}^m$ is called the span of $x_1, \ldots , x_k$ and is denoted by $\textnormal{span}(x_1, \ldots , x_k)$. A basis of some subspace $\mathcal{X} \subseteq \mathbb{K}^m$ is a linearly independent set of vectors $x_1, \ldots , x_k \in \mathcal{X}$ such that $\mathcal{X} = \textnormal{span}(x_1, \ldots , x_k)$. In this case we say that the dimension of $\mathcal{X}$ equals $k$, that is, $\dim(\mathcal{X}) = k$.  The symbol $\mathbb{K}^m \times \mathbb{K}^m$ is used to denote the direct product of $\mathbb{K}^m$ with itself, i.e. $\mathbb{K}^m \times \mathbb{K}^m = \lbrace (x,y) \; | \; x,y \in \mathbb{K}^m \rbrace$. For any matrix $A \in \MM_{m \times m}(\mathbb{K})$, the notions $\textnormal{im}(A)$ and $\textnormal{null}(A)$ refer to the image and the nullspace (kernel) of $A$, i.e. $\textnormal{im}(A) = \lbrace Ax \, | \, x \in \mathbb{K}^m \rbrace$ and $\textnormal{null}(A) = \lbrace x \in \mathbb{K}^m \, | \, Ax = 0 \rbrace$. The rank of $A \in \MM_{m \times m}(\mathbb{K})$ is defined as the dimension of its image. For any matrix $A \in \MM_{m \times m}(\mathbb{K})$ the superscripts $T$ and $H$ denote the transpose $A^T$ of $A$ and the Hermitian transpose $A^H = \overline{A}^T$. The overbar denotes the conjugation of a complex number and applies entrywise to matrices.  The $m \times m$ identity matrix is throughout denoted by $I_m$ whereas the $m \times m$ zero matrix, the zero vector in $\mathbb{K}^m$ or the number zero are simply denoted by $0$ (to specify dimensions $0_{m \times m}$ is used in some places to refer to the $m \times m$ zero matrix). A Hermitian matrix $A \in \MM_{m \times m}(\mathbb{K})$ satisfies $A^H=A$ and a skew-Hermitian matrix $A^H = -A$. Moreover, a matrix $A \in \MM_{m \times m}(\mathbb{K})$ is called unitary if $A^HA = AA^H= I_m$ holds and normal in case $A^HA = AA^H$ holds. For two matrices $A,B \in \MM_{m \times m}(\mathbb{K})$ the notation $A \oplus B$ is used to denote their direct sum, i.e. the matrix $C \in \MM_{2m \times 2m}(\mathbb{K})$ given by
$$ C = \begin{bmatrix} A & 0_{m \times m} \\ 0_{m \times m} & B \end{bmatrix}.$$
For a given matrix $A \in \MM_{m \times m}(\mathbb{K})$ any scalar $\lambda \in \CC$ which satisfies $Ax = \lambda x$ for some nonzero vector $x \in \CC^m$ is called an eigenvalue of $A$. The set of all eigenvalues of $A$ is denoted by $\sigma(A)$ and equals the zero set of the degree-$m$ polynomial $\det(A - zI_m)$. The algebraic multiplicity of $\lambda$ as an eigenvalue of $A$ equals the multiplicity of $\lambda$ as a zero of $\det(A-zI_m)$.  Whenever $\lambda \in \CC$ is some eigenvalue of $A$ any vector $x \in \CC^m$ satisfying $Ax=\lambda x$ is called an eigenvector of $A$ (for $\lambda$). The set of all eigenvectors of $A$ for $\lambda \in \sigma(A)$ is a vector subspace of $\CC^m$ and is called the corresponding eigenspace (of $A$ for $\lambda$). Its dimension is referred to as the geometric multiplicity of $\lambda$. The matrix $A$ is called diagonalizable if there exist $m$ linearly independent eigenvectors of $A$. These vectors consequently form a basis of $\CC^m$. A matrix $A \in \MM_{m \times m}(\mathbb{K})$ is diagonalizable if and only if the geometric and algebraic multiplicities of all eigenvalues of $A$ coincide.

\section{Sesquilinear Forms}
\label{sec:sesf}

In this section we introduce the notion of a sesquilinear form on $\CC^m \times \CC^m$ and some related basic concepts. Notice that Definition \ref{def:bilinearform} slightly deviates from the definition of a sesquilinear form given in \cite[Sec.\,5.1]{Kub11}.

\begin{definition}\label{def:bilinearform}
A  sesquilinear form $[ \cdot , \cdot]$ on $\CC^{m} \times \CC^m$ is a mapping $[\cdot , \cdot]: \CC^m \times \CC^m \rightarrow \CC$ so that for all $u,v,w \in \CC^m$ and all $\alpha , \beta \in \CC$ the following relations $(i)$ and $(ii)$ hold
\begin{align*} (i) \; [\alpha u + \beta v , w] = \overline{\alpha} [u,w] + \overline{\beta} [v,w] \qquad (ii) \; [u, \alpha v + \beta w ] &= \alpha [u,v] + \beta [u,w]. \end{align*}
\end{definition}

If $[ \cdot, \cdot]$ is some sesquilinear form and $x := \alpha e_j, y := \beta e_k \in \CC^m$ with $\alpha , \beta \in \CC$ are two vectors that are multiples of the $j$th and $k$th unit vectors $e_j$ and $e_k$, then $[x,y] = \overline{\alpha} \beta [e_j,e_k]$. Thus any sesquilinear form is uniquely determined by the images of the standard unit vectors $[e_j,e_k], j,k=1, \ldots , m$. In particular, $[ \cdot , \cdot]$ on $\CC^m \times \CC^m$ can be expressed as \begin{equation} [x,y]=x^HBy \label{bilinearform} \end{equation} for the particular matrix $B = [b_{jk}]_{jk} \in \MM_{m \times m}(\CC)$ with $b_{jk} = [e_j,e_k]$, $j,k=1, \ldots , m$.
A form $[ \cdot , \cdot]$  as in (\ref{bilinearform}) is called Hermitian if $[x,y]=\overline{[y,x]}$ holds for all $x,y \in \CC^m$. It is easy to see that $[ \cdot , \cdot]$ is Hermitian \index{matrix!Hermitian} \index{sesquilinear form!Hermitian}  if and only if $B \in \MM_{m \times m}(\CC)$ is Hermitian, i.e. $B=B^H$ \cite[Sec.\,2.1]{GoLaRod05}. The form $[ \cdot, \cdot]$ is called skew-Hermitian if $[x,y]=- \overline{[y,x]}$ holds for all $x,y \in \CC^m$. This is the case if and only if $B=-B^H$.

The following Definition \ref{def:gramian} introduces two classes of subspaces $\mathcal{S} \subseteq \CC^m$ related in a particular fashion to a sesquilinear form $[ \cdot, \cdot]$ (see, e.g., \cite[Sec.\,2.3]{GoLaRod05}).

\begin{definition} \label{def:gramian} Let $[x,y]=x^HBy$ be some sesquilinear form on $\CC^{m} \times \CC^{m}$.
\begin{enumerate}
\item A subspace $\mathcal{S} \subseteq \CC^{m}$ of dimension $\dim (\mathcal{S}) = k \geq 1$ is called neutral (with respect to $[\cdot , \cdot]$) if $\textnormal{rank}(V^HBV) = 0$ for any basis $v_1 , \ldots , v_k$ of $\mathcal{S}$ and $V = [ \, v _1 \; \cdots \; v_k \, ].$
\item A subspace $\mathcal{S} \subseteq \CC^{m}$ of dimension $\dim (\mathcal{S}) = k \geq 1$ is called nondegenerate (with respect to $[\cdot , \cdot]$) if $V^HBV$ is nonsingular, i.e. $\textnormal{rank}(V^HBV)=k$, for any basis $v_1 , \ldots , v_k$ of $\mathcal{S}$ and $V = [ \, v _1 \; \cdots \; v_k \, ].$ Otherwise, $\mathcal{S}$ is called degenerate.
\end{enumerate}
\end{definition}

In case $m=2n$ is even, any neutral subspace $\mathcal{S} \subseteq \CC^{m}$ with $\dim(\mathcal{S}) = n$ is called \textit{Lagrangian} (subspace) (see, e.g., \cite[Def.\,1.2]{Freiling02}). Some analysis on this kind of subspaces is presented in Section \ref{sec:Lagrangian}. A sesquilinear form as in \eqref{bilinearform} is called \textit{nondegenerate}, if $\mathcal{S} = \CC^{m}$ is nondegenerate with respect to $[\cdot, \cdot]$. In the sequel, nondegenerate sesquilinear forms are called \textit{indefinite inner products}.  Note that the sesquilinear form in \eqref{bilinearform} is nondegenerate, i.e. an indefinite inner product, if and only if $B \in \MM_{m \times m}(\CC)$ is nonsingular \cite[Sec.\,2.1]{Mack05}.

\begin{proposition} \label{prop:adjoint}
For any indefinite inner product $[x,y]=x^HBy$ on $\CC^{m} \times \CC^{m}$ and any $A \in \MM_{m \times m}(\CC)$ there exists a unique matrix $A^\star \in \MM_{m \times m}(\CC)$  such that
$$ [Ax,y]= [x, A^\star y] \quad \textnormal{holds for all } x,y \in \CC^{m}.$$
\end{proposition}
The matrix $A^\star \in \MM_{m \times m}(\CC)$ corresponding to $A \in \MM_{m \times m}(\CC)$ in Proposition \ref{prop:adjoint} is called the \textit{adjoint} of $A$. It can be expressed as $A^\star = B^{-1}A^HB$ and also satisfies $[x,Ay] = [A^\star x,y]$ for all $x,y \in \CC^{m}$ \cite[Sec.\,2.2]{Mack05}. A matrix $A$ that commutes with its adjoint $A^\star$, i.e. $AA^\star = A^\star A$, is called \textit{normal with respect to} $[x,y]=x^HBy$ or simply $B$\textit{-normal}. For any indefinite inner product $[x,y]=x^HBy$ on $\CC^{m} \times \CC^{m}$ there are three classes of $B$-normal matrices that deserve special attention (see also \cite[Sec.\,2.2]{Mack05}).

\begin{definition}\label{def:classes}
Let $[x,y]=x^HBy$ be some indefinite inner product on $\CC^{m} \times \CC^{m}$.
\begin{enumerate}
\item A matrix $G \in \MM_{m \times m}(\CC)$ with the property $G^{-1} = G^\star$ is called an automorphism (with respect to $[ \cdot, \cdot]$).
\item A matrix $J \in \MM_{m \times m}(\CC)$ satisfying $J^\star =B^{-1}J^HB = J$ is called selfadjoint (with respect to $[\cdot, \cdot]$) whereas a matrix $L \in \MM_{m \times m}(\CC)$  satisfying $L^\star = B^{-1}L^HB = - L$ is called skewadjoint.
\end{enumerate}
\end{definition}

Notice that, if $G \in \MM_{m \times m}(\CC)$ is an automorphism, $[Gx,Gy]=[x,y]$ holds for all $x,y \in \CC^{m}$ since $[Gx,Gy] = [x,G^\star G y]$ and $G^\star G = G^{-1} G = I_m$.
%implies $[x,y]=[G^\star Gx,y]$ for all $x,y \in \CC^{m}$. This in turn can only hold if $G^\star G = I_{m}$, i.e. $G^{-1} = G^\star$.
In particular, any automorphism is nonsingular. For the standard Euclidean scalar product $(x,y)=[x,y]=x^HI_{m}y$, automorphisms, selfadjoint and skewadjoint matrices are those which are unitary, Hermitian or skew-Hermitian, respectively. Beside these, special names have also been given to matrices which are automorph, selfadjoint or skewadjoint with respect to the indefinite inner products $[x,y]=x^HBy$ on $\CC^{2n} \times \CC^{2n}$ induced by the matrices $B = J_{2n} \in \MM_{2n \times 2n}(\mathbb{R})$ and $B = R_{2n} \in \MM_{2n \times 2n}(\mathbb{R})$ given by
$$ J_{2n} = \begin{bmatrix} & I_n \\ -I_n & \end{bmatrix}, \quad R_{2n} = \begin{bmatrix} & R_n \\ R_n & \end{bmatrix} \; \textnormal{with} \;  R_n = \begin{bmatrix} & & 1 \\ & \iddots & \\ 1 & & \end{bmatrix}.$$
 These names are listed in the table from Figure \ref{fig:1}\footnote{Notice that these names are not consistently used in the literature. For instance, a Hamiltonian matrix here \\ and in \cite{Freiling02} is called $J$-Hermitian in \cite{Mack05}.}. For instance, a skew-Hamiltonian matrix $A \in \MM_{2n \times 2n}(\CC)$ and a per-Hermitian matrix $C \in \MM_{2n \times 2n}(\CC)$ have expressions of the form
\begin{equation} A = \begin{bmatrix} A_1 & A_2 \\ A_3 & A_1^H \end{bmatrix}  \quad \textnormal{and} \quad C = \begin{bmatrix} C_1 & C_2 \\ C_3 & R_nC_1^HR_n \end{bmatrix}, \quad A_j, C_j \in \MM_{n \times n}(\CC), \label{equ:ex_matrices} \end{equation}
where it holds that $A_2=-A_2^H, A_3=-A_3^H$ and that $C_2, C_3 \in \MM_{n \times n}(\CC)$ are themselves per-Hermitian with respect to $[x,y]=x^HR_ny$ on $\CC^n \times \CC^n$.  Notice that for any indefinite inner product  $[x,y]=x^HBy$ on $\CC^{m} \times \CC^{m}$ the selfadjoint and skewadjoint structures are preserved under similarity transformations with automorphisms. This fact is well known and easily confirmed for unitary similarity transformations of Hermitian and skew-Hermitian matrices. In our setting this means that, whenever $A \in \MM_{2n \times 2n}(\CC)$ is (skew)-Hamiltonian (per(skew)-Hermitian)  and $G \in \MM_{2n \times 2n}(\CC)$ is symplectic (perplectic), then $G^{-1}AG$ is again (skew)-Hamiltonian (per(skew)-Hermitian). We will only be considering the indefinite inner products induced by $J_{2n}$ and $R_{2n}$ on $\CC^{2n} \times \CC^{2n}$ from now on.

\begin{table}
\caption{Structures with respect to $[x,y]=x^HJ_{2n}y$ and $[x,y]=x^HR_{2n}y$ on $\CC^{2n} \times \CC^{2n}$.}
\begin{center} \small
\begin{tabular}{|l||l|l||l|l|} \hline
Structure &   $[x,y]=x^HJ_{2n}y$ & &   $[x,y]=x^HR_{2n}y$ &\\ \hline \hline
selfadjoint & skew-Hamiltonian & $J_{2n}^TA^HJ_{2n} = A$ & per-Hermitian & $R_{2n}A^HR_{2n} = A$ \\
skewadjoint & Hamiltonian & $J_{2n}^TA^HJ_{2n} = -A$ & perskew-Hermitian & $R_{2n}A^HR_{2n} = -A$ \\ \hline
automorph & symplectic & $J_{2n}^TA^HJ_{2n} = A^{-1}$ & perplectic & $R_{2n}A^HR_{2n} = A^{-1}$ \\ \hline
\end{tabular}
\end{center}
\label{fig:1}
\end{table}
\normalsize
The result from Proposition \ref{prop:sylvester} below is central for the upcoming discussion and can be found in, e.g., \cite[Sec.\,4.5]{HoJo90} (for the case $A=A^H$). The statement for $A=-A^H$ is easily verified by noting that $A = A^H$ is Hermitian if and only if $iA$ is skew-Hermitian.

\begin{proposition}[Sylvesters Law of Inertia] \label{prop:sylvester}  Let $A \in \MM_{m \times m}(\CC)$ and assume that either $A=A^H$ or $A=-A^H$ holds. Then there exists a nonsingular matrix $U \in \MM_{m \times m}(\CC)$ so that
$$ U^H A U = \left[ \begin{array}{c|c|c} - \alpha I_p & & \\ \hline & \alpha I_q & \\ \hline & & 0_{r \times r} \end{array} \right]$$
where $\alpha = 1$ if $A$ is Hermitian and $\alpha = i$ otherwise. Hereby, $p$ coincides with the number of negative real/purely imaginary eigenvalues of $A$, $q$ coincides with the number of positive real/purely imaginary eigenvalues of $A$ and $r$ is the algebraic multiplicity of zero as an eigenvalue of $A$.
\end{proposition}

The triple $(p,q,r)$ from Proposition \ref{prop:sylvester} is usually referred to as the inertia of $A$ \cite[Sec.\,4.5]{HoJo90}. Two Hermitian or skew-Hermitian matrices $A,C \in \MM_{m \times m}(\CC)$ with the same inertia are called \textit{congruent}.  Following directly from Proposition \ref{prop:sylvester} we obtain the following proposition (see also \cite[Thm.\,4.5.8]{HoJo90}).

\begin{proposition} \label{prop:congruence}
Let  $A,C \in \MM_{m \times m}(\CC)$ be two matrices which are either both Hermitian or skew-Hermitian. Then there exists a nonsingular matrix $S \in \MM_{n \times n}(\CC)$ so that $S^HAS=C$ if and only if $A$ and $C$ have the same inertia.
\end{proposition}

\section{Symplectic and Perplectic Diagonalizability}
\label{sec:diag}

In this section the symplectic and perplectic diagonalization of (skew)-Hamiltonian  and per(skew)-Hermitian matrices is analyzed. As those matrices need not be diagonalizable per se, cf. \cite[Ex.\,4.2.1]{GoLaRod05}, their diagonalizability has to be assumed throughout the whole section. At first, we consider arbitrary (skew)-Hermitian indefinite inner products and provide two auxiliary results related to their selfadjoint matrices. These results will turn out to be useful in Sections \ref{sec:symp_diag} and \ref{sec:perp_diag} where we derive necessary and sufficient conditions for (skew)-Hamiltonian  or per(skew)-Hermitian matrices to be diagonalizable by a symplectic (perplectic, respectively) similarity transformation. This section is based on \cite[Chap.\,9]{PS2019}.

Let $[x,y]=x^HBy$ be some (skew)-Hermitian indefinite inner product on $\CC^{m} \times \CC^{m}$ and let $A \in \MM_{m \times m}(\CC)$ be selfadjoint with respect to $[ \cdot, \cdot]$. Then, as $A^\star = B^{-1}A^HB = A$, we have $\sigma(A) = \overline{\sigma(A)}$. In particular, for each $\lambda \in \sigma(A)$, $\lambda \notin \mathbb{R}$, $\overline{\lambda}$ is an eigenvalue of $A$, too, with the same multiplicity. Proposition \ref{prop:ortheigenspaces} shows that, among the eigenvectors of $A$, those $x,y \in \CC^{m}$ corresponding to $\lambda$ and $\overline{\lambda}$, respectively, are the only candidates for having a nonzero inner product $[x,y]$. This result can also be found in, e.g., \cite[Thm.\,7.8]{Mack05}.

\begin{proposition} \label{prop:ortheigenspaces}
Let $[x,y]=x^HBy$ be some (skew)-Hermitian indefinite inner product and $A = A^\star \in \MM_{2n \times 2n}(\CC)$ be selfadjoint. Moreover, assume $x,y \in \CC^{2n}$ are eigenvectors of $A$ corresponding to some eigenvalues $\lambda, \mu \in \sigma(A)$, respectively. Then $\mu \neq \overline{\lambda}$ implies that $[x,y]=[y,x]=0$. Consequently, each eigenspace of $A$ for an eigenvalue $\lambda \notin \mathbb{R}$ of $A$ is neutral.
\end{proposition}

\begin{proof}
Under the given assumptions we have
$$ \overline{\lambda} [x,y] = [ \lambda x, y] = [Ax,y] = [x, A^\star y] = [x,Ay] = [x, \mu y] = \mu [x,y]$$ so, whenever $[x,y] \neq 0$ then $\mu = \overline{\lambda}$ has to hold. This proves the statement by contraposition noting that $[x,y]=0$ if and only if $[y,x]=0$.
\end{proof}

Now assume that $A =A^\star \in \MM_{m \times m}(\CC)$ is diagonalizable. Let $\lambda \in \sigma(A)$, $\lambda \neq \overline{\lambda}$, and suppose $v_1, \ldots , v_\ell$ and $v_{\ell+1}, \ldots , v_{2\ell}$ are eigenbases corresponding to $\lambda$ and $\overline{\lambda}$, respectively. Additionally, let $v_{2\ell+1}, \ldots , v_{m}$ be eigenvectors of $A$ completing $v_1, \dots , v_{2\ell}$ to a basis of $\CC^{m}$ and set $V = [ \, v_1 \; \cdots \; v_{m} \, ] \in \MM_{m \times m}(\CC)$. According to Proposition \ref{prop:ortheigenspaces} we have
\begin{equation} V^H B V = \left[ \begin{array}{c|c} \begin{array}{c|c} 0 & S_\ell \\ \hline  \pm S_\ell^H & 0 \end{array} & 0 \\ \hline 0 & X \end{array} \right] \in \MM_{m \times m}(\CC)\label{equ:nonsingular} \end{equation}
for some matrices $S_\ell \in \MM_{\ell \times \ell}(\CC)$ and $X \in \MM_{(m-2 \ell) \times (m - 2 \ell)}(\CC)$. In case $B=-B^H$ we have $- S_\ell^H$ in \eqref{equ:nonsingular} and $X=-X^H$ whereas we have $+ S_\ell$ and $X=X^H$ in case $B=B^H$. As $V$ and $B$ are nonsingular, so is $V^HBV$. This implies $S_\ell$ and $X$ in \eqref{equ:nonsingular} to be nonsingular, too. As $\textnormal{span}(v_1, \ldots , v_{2\ell})$ equals the direct sum of the eigenspaces of $A$ corresponding to $\lambda$ and $\overline{\lambda}$, the nonsingularity of $S_\ell$ gives the following Corollary \ref{cor:nondegeigen} taking Definition \ref{def:gramian} (1) into account.

\begin{corollary} \label{cor:nondegeigen}
Let $[x,y]=x^HBy$ be some (skew)-Hermitian indefinite inner product and let $A = A^\star \in \MM_{m \times m}(\CC)$ be selfadjoint and diagonalizable. Then, for any $\lambda \in \sigma(A)$, $\lambda \neq \overline{\lambda}$, the direct sum of the eigenspaces of $A$ corresponding to $\lambda$ and $\overline{\lambda}$ is always nondegenerate.
\end{corollary}

Similarly to the derivation preceding Corollary \ref{cor:nondegeigen} one shows that the eigenspace of a selfadjoint matrix $A =A^\star \in \MM_{m \times m}(\CC)$ corresponding to some real eigenvalue $\mu \in \sigma(A)$ is always nondegenerate, too. We are now in the position to derive statements on the symplectic and perplectic diagonalizability of (skew)-Hamiltonian and per(skew)-Hermitian matrices.

\subsection{Symplectic Diagonalization of (skew)-Hamiltonian Matrices}
\label{sec:symp_diag}

The following Theorem \ref{thm:diag_sympl0} states the main result of this section characterizing those (diagonalizable) (skew)-Hamiltonian matrices which can be brought to diagonal form by a symplectic similarity transformation. Recall that, according to \eqref{equ:ex_matrices}, a diagonal skew-Hamiltonian matrix $\widetilde{D} \in \textnormal{M}_{2n \times 2n}(\CC)$  has the form
\begin{equation} \widetilde{D} = \begin{bmatrix} D & 0 \\0 &  D^H \end{bmatrix} \quad \textnormal{with} \; D = \textnormal{diag}(\, \lambda_1, \ldots , \lambda_n \, ) \in \textnormal{M}_{n \times n}(\CC). \label{equ:diagHamiltonian} \end{equation}

\begin{theorem} \label{thm:diag_sympl0}
Let $A \in \MM_{2n \times 2n}(\CC)$ be diagonalizable.
\begin{enumerate}
\item Assume that $A$ is skew-Hamiltonian. Then $A$ is symplectic diagonalizable if and only if for any real eigenvalue $\lambda \in \sigma(A)$ and some basis $v_1, \ldots , v_{m}$ of the corresponding eigenspace, the matrix $V^HJ_{2n}V$ for $V = [ \, v_1 \; \cdots \; v_{m} \, ]$ has equally many positive and negative imaginary eigenvalues.
\item Assume that $A$ is Hamiltonian. Then $A$ is symplectic diagonalizable if and only if for any purely imaginary eigenvalue $\lambda \in \sigma(A)$ and some basis $v_1, \ldots , v_{m}$ of the corresponding eigenspace, the matrix $V^HJ_{2n}V$ for $V = [ \, v_1 \; \cdots \; v_{m} \, ]$ has equally many positive and negative imaginary eigenvalues.
\end{enumerate}
\end{theorem}

\begin{proof}
1. $\Rightarrow$ Let $A \in \MM_{2n \times 2n}(\CC)$ be skew-Hamiltonian, that is $A=A^\star$, and $S =[ \, s_1 \; \cdots \; s_{2n} \, ] \in \MM_{2n \times 2n}(\CC)$ symplectic such that
\begin{equation} S^{-1}AS = \begin{bmatrix} D & \\ & D^H \end{bmatrix} = S^{-1}A^\star S, \quad S^HJ_{2n}S = J_{2n}, \label{equ:sympdiag1} \end{equation}
with $D = \textnormal{diag}(\lambda_1, \ldots , \lambda_n) \in \MM_{n \times n}(\CC)$ is a (symplectic) diagonalization of $A$. If $\lambda_j \in \sigma(A)$ is real, it follows from \eqref{equ:sympdiag1} that $\lambda_j$ has even multiplicity, $2k$ say, with $k$ instances of $\lambda_j$ appearing in $D$ and $D^H$, respectively (w.\,l.\,o.\,g. on the diagonal positions $j_1, \ldots , j_k$). Let $s_{j_1}, \ldots , s_{j_k}, s_{n+j_1}, \ldots , s_{n+j_k}$ be the corresponding $2k$ eigenvectors (appearing as columns in the corresponding positions in $S$) which span the eigenspace of $A$ and $A^\star$ for $\lambda_j$. Now set $S_j := [ \, s_{j_1} \, \cdots \, s_{j_k} \, s_{n+j_1} \, \cdots \, s_{n+j_k} \, ] \in \MM_{2n \times 2k}(\CC)$. Then we have
$$ S_j^HJ_{2n}S_j = \begin{bmatrix} & I_k \\ -I_k & \end{bmatrix} \in \MM_{2k \times 2k}(\CC)$$
which follows directly from $S^HJ_{2n}S = J_{2n}$. The eigenvalues of $S_j^HJ_{2n}S_j$ are $+i$ and $-i$ both with the same multiplicity $k$. As $\lambda_j$ was arbitrary, this holds for any real eigenvalue of $A$.

$\Leftarrow$ Now let $A \in \MM_{2n \times 2n}(\CC)$ be skew-Hamiltonian and diagonalizable. Moreover assume that the condition stated above holds for all real eigenvalues of $A$. We now generate bases for the different eigenspaces of $A$ according to the following rules:
\begin{enumerate}
\item[(a)] For each pair of eigenvalues $\lambda_j, \overline{\lambda_j} \in \sigma(A)$, $\lambda_j \neq \overline{\lambda_j}$, both with multiplicity $m_j$, let $s_1, \ldots s_{m_j}$ be corresponding eigenvectors of $A$ for $\lambda_j$ and $t_1, \ldots , t_{m_j}$ corresponding eigenvectors of $A$ for $\overline{\lambda_j}$. Set $S_j = [ \, s_1 \, \cdots \, s_{m_j} \, t_1 \, \cdots \, t_{m_j} \, ] \in \MM_{2n \times 2m_j}(\CC)$. Then, according to Proposition \ref{prop:ortheigenspaces} and Corollary \ref{cor:nondegeigen} $\textnormal{span}(s_1, \ldots , s_{m_j})$ and $\textnormal{span}(t_1, \ldots , t_{m_j})$ are both neutral and $S_j^HJ_{2n}S_j$ is nonsingular. Therefore the form of $S_j^HJ_{2n}S_j$ is
    $$ S_j^HJ_{2n}S_j = \begin{bmatrix} 0 & \widehat{S}_j \\ - \widehat{S}_j^H & 0 \end{bmatrix} \in \MM_{2m_j \times 2m_j}(\CC)$$
    for some nonsingular matrix $\widehat{S}_j \in \MM_{m_j \times m_j}(\CC)$. Now, multiplying $S_j^HJ_{2n}S_j$ by $\widehat{S}_j^{-H} \oplus I_{m_j}$ and $(\widehat{S}_j^{-H} \oplus I_{m_j})^H$ (from the right and the left) we observe that
     \begin{align*}\big( \widehat{S}_j^{-H} \oplus I_{m_j} \big)^H S_j^HJ_{2n}S_j \big( \widehat{S}_j^{-H} \oplus I_{m_j} \big) &= \begin{bmatrix} \widehat{S}_j^{-1} & \\ & I_{m_j} \end{bmatrix} \begin{bmatrix} 0 & \widehat{S}_j \\ - \widehat{S}_j^H  & 0 \end{bmatrix} \begin{bmatrix} \widehat{S}_j^{-H} & \\ & I_{m_j} \end{bmatrix} \\ &= \begin{bmatrix} \vphantom{\widehat{S}_j} & I_{m_j} \\ -I_{m_j} & \vphantom{\widehat{S}_j} \end{bmatrix}. \end{align*}

    Let $w_1, \ldots , w_{2m_j}$ denote the columns of $S_j(\widehat{S}_j^{-H} \oplus I_{m_j})$ and notice that, due to the form of $\widehat{S}_j^{-H} \oplus I_{m_j}$, $w_1, \ldots , w_{m_j}$ and $w_{m_j+1} =t_1, \ldots , w_{2m_j} = t_{m_j}$ are still bases for the eigenspaces of $A$ for $\lambda_j$ and $\overline{\lambda_j}$, respectively. According to Proposition \ref{prop:ortheigenspaces}, the inner products $[w_\ell, x]$ for any $\ell=1, \ldots , 2m_j$ and any eigenvector $x$ of $A$ corresponding to some eigenvalue $\mu \in \sigma(A) \setminus \lbrace \lambda_j , \overline{\lambda_j} \rbrace$ are zero.
\item[(b)] For each $\lambda_k \in \sigma(A)$, $\lambda_k \in \mathbb{R}$,  let $s_1, \ldots , s_{2m_k}$ be a basis of the corresponding eigenspace (assuming the even multiplicity of $\lambda_k$ is $2 m_k$). For $S_k := [ \, s_1 \; \cdots \, s_{2m_k} \, ] \in \MM_{2n \times 2m_k}(\CC)$ the skew-Hermitian matrix $S_k^HJ_{2n}S_k \in \MM_{2m_k \times 2m_k}(\CC)$ is nonsingular and has, according to our assumptions, exactly $m_k$ positive and $m_k$ negative purely imaginary eigenvalues. Thus, it has the same inertia as $J_{2m_k}$ and there exists some nonsingular matrix $T_k \in \MM_{2m_k \times 2m_k}(\CC)$ such that $T_k^H(S_k^HJ_{2n}S_k)T_k = J_{2m_k}$ according to Proposition \ref{prop:sylvester}. Let $w_1, \ldots , w_{2m_k}$ denote the columns of $S_kT_k$ and note that $w_1, \ldots , w_{2m_k}$ is still a basis for the eigenspace of $A$ corresponding to $\lambda_k$. According to Proposition \ref{prop:ortheigenspaces}, the inner products $[w_\ell, x]$ for any $\ell=1, \ldots , 2m_k$ and any eigenvector $x$ of $A$ corresponding to some eigenvalue $\mu \in \sigma(A) \setminus \lbrace \lambda_k \rbrace$ are zero.
\end{enumerate}

If bases of the eigenspaces for all eigenvalues of $A$ have been constructed according to (a) if $\lambda_j \notin \mathbb{R}$ and (b) if $\lambda_k \in \mathbb{R}$, the new eigenvectors $w_1, \ldots , w_{2n}$ obtained this way are collected in a matrix $W \in \MM_{2n \times 2n}(\CC)$, i.e. $W = [ \, w_1 \; \cdots \; w_{2n} \, ].$ Note that $W$ is nonsingular and that $W^{-1}A W = D$ is diagonal. Due to the construction of $w_1, \ldots , w_{2n}$, the skew-Hermitian matrix $W^HJ_{2n} W$ has only $+1$ and $-1$ as nonzero entries. Hence, it is permutation-similar to $J_{2n}$. In other words, there exists a (real) permutation matrix $P \in \MM_{2n \times 2n}(\mathbb{R})$ with $P^H W^HJ_{2n} W P = J_{2n}$. Now $P^H W^HJ_{2n} WP = J_{2n}$ so $V := WP$ is symplectic. Moreover, $V^{-1}AV = P^TDP$ remains to be diagonal as $P$ is a permutation matrix and the statement 1. is proven.

2. If $A \in \MM_{2n \times 2n}(\CC)$ is Hamiltonian notice that $\widehat{A} := iA$ is skew-Hamiltonian. Thus, whenever $S \in \MM_{2n \times 2n}(\CC)$ is symplectic and $S^{-1}AS = D \oplus (-D^H)$ is a symplectic diagonalization of $A$ for some diagonal matrix $D \in \MM_{n \times n}(\CC)$ we have that
$$ S^{-1} \widehat{A}S = S^{-1}(iA)S = i S^{-1}AS = \begin{bmatrix} iD & \\ & -iD^H \end{bmatrix} = \begin{bmatrix} \widehat{D} & \\ & \widehat{D}^H \end{bmatrix}$$
for $\widehat{D} = iD$ is a symplectic diagonalization of $\widehat{A}$. From 1. it is known that the diagonalization $S^{-1} \widehat{A}S = \widehat{D} \oplus \widehat{D}^H$ exists if and only if for each real eigenvalue $\lambda \in \sigma(\widehat{A})$ has even multiplicity $m$ and, given any basis $v_1, \ldots , v_{m}$ of the corresponding eigenspace, the matrix $V^HJ_{2n}V$ for $V=[ \, v_1 \; \cdots \; v_m \, ]$ has equally many positive and negative purely imaginary eigenvalues. Vice versa this implies that the symplectic diagonalization $S^{-1}AS = D \oplus (-D^H)$ exists if and only if each purely imaginary eigenvalue $\mu \in \sigma(A)$ has even multiplicity $m$ and, given any basis $v_1, \ldots , v_{m}$ of the corresponding eigenspace, the matrix $V^HJ_{2n}V$ for $V=[ \, v_1 \; \cdots \; v_m \, ]$ has equally many positive and negative purely imaginary eigenvalues.
\end{proof}

The following Corollary \ref{cor:norealeigenvalues} is a direct consequence of Theorem \ref{thm:diag_sympl0} which guarantees the existence of a symplectic diagonalization whenever no real or purely imaginary eigenvalues are present. To understand Corollary \ref{cor:norealeigenvalues} correctly, zero should be regarded as both, real and purely imaginary.

\begin{corollary} \label{cor:norealeigenvalues}
\begin{enumerate}
\item A diagonalizable skew-Hamiltonian  matrix $A \in \MM_{2n \times 2n}(\CC)$ is always symplectic diagonalizable if $A$ has no purely real eigenvalues.
\item A diagonalizable Hamiltonian  matrix $A \in \MM_{2n \times 2n}(\CC)$ is always symplectic  diagonalizable if $A$ has no purely imaginary eigenvalues.
\end{enumerate}
\end{corollary}

\begin{example}
Let $A,B \in \MM_{n \times n}(\CC)$ be skew-Hermitian matrices. Taking \eqref{equ:ex_matrices} into account it is easy to check that the matrix
\begin{equation} M = \begin{bmatrix} A & B \\ B & -A \end{bmatrix} \in \MM_{2n \times 2n}(\CC) \label{ex:HamHerm} \end{equation}
is skew-Hamiltonian and skew-Hermitian. The skew-Hermitian structure implies that $M$ has only purely imaginary eigenvalues. Therefore, Corollary \ref{cor:norealeigenvalues} applies and, whenever $M$ is nonsingular, it can be diagonalized by a symplectic similarity transformation. The diagonalizability of $M$ is always guaranteed since any skew-Hermitian matrix can be diagonalized (by a unitary matrix). In Section \ref{sec:norm_symp} we will show that a symplectic diagonalization of $M$ can always be constructed to be unitary, too. An analogous statement holds for nonsingular matrices $M \in \MM_{2n \times 2n}(\CC)$ of the form \eqref{ex:HamHerm} where $A$ and $B$ are Hermitian. Such matrices are consequently Hermitian and Hamiltonian, i.e. they have no purely imaginary eigenvalues.
\end{example}

\subsection{Perplectic Diagonalization of per(skew)-Hermitian Matrices}
\label{sec:perp_diag}

The main result on the perplectic diagonalization of per-Hermitian and perskew-Hermitian matrices $A \in \MM_{2n \times 2n}(\CC)$ is similar to the statement from Theorem \ref{thm:diag_sympl0}. In particular, the proof of Theorem \ref{thm:diag_perpl0} below is analogous to the proof of Theorem \ref{thm:diag_sympl0}  with the only significant change being the replacement of the skew-Hermitian structures  appearing in the proof of Theorem \ref{thm:diag_sympl0} (due to the skew-Hermitian matrix $J_{2n}$) by Hermitian structures caused by $R_{2n}$. Therefore, statements on purely imaginary eigenvalues turn into statements on real eigenvalues. The proof is consequently omitted.

\begin{theorem} \label{thm:diag_perpl0}
Let $A \in \MM_{n \times n}(\CC)$ be diagonalizable.
\begin{enumerate}
\item Assume that $A$ is per-Hermitian. Then $A$ is perplectic diagonalizable if and only if for any real eigenvalue $\lambda \in \sigma(A)$ and some basis $v_1, \ldots , v_{m}$ of the corresponding eigenspace, the matrix $V^HR_{2n}V$ for $V = [ \, v_1 \; \cdots \; v_{m} \, ]$ has equally many positive and negative real eigenvalues.
\item Assume that $A$ is perskew-Hermitian. Then $A$ is perplectic diagonalizable if and only if for any purely imaginary eigenvalue $\lambda \in \sigma(A)$ and some basis $v_1, \ldots , v_{m}$ of the corresponding eigenspace, the matrix $V^HR_{2n}V$ for $V = [ \, v_1 \; \cdots \; v_{m} \, ]$ has equally many positive and negative real eigenvalues.
\end{enumerate}
\end{theorem}

The following Corollary \ref{cor:norealeigenvalues1} is an immediate consequence of Theorem \ref{thm:diag_perpl0} and is the analogous result to Corollary \ref{cor:norealeigenvalues} for per(skew)-Hermitian matrices.

\begin{corollary} \label{cor:norealeigenvalues1}
\begin{enumerate}
\item A diagonalizable per-Hermitian matrix $A \in \MM_{2n \times 2n}(\CC)$ is always perplectic diagonalizable if $A$ has no purely real eigenvalues.
\item A diagonalizable perskew-Hermitian matrix $A \in \MM_{2n \times 2n}(\CC)$ is always perplectic diagonalizable if $A$ has no purely imaginary eigenvalues.
\end{enumerate}
\end{corollary}

\section{Normal Structured Matrices}
\label{sec:normal}
In this section we analyze the matrix structures from Section \ref{sec:diag} assuming the matrix at hand is additionally normal.
Recall that a matrix $A$ is called normal if $A^HA = AA^H$ holds. It is well-known that for any normal matrix $A \in \MM_{2n \times 2n}(\CC)$ there exists a unitary matrix $Q \in \textnormal{M}_{2n \times 2n}(\CC)$, so that $Q^HAQ = D = \textnormal{diag}(\lambda_1, \ldots , \lambda_{2n})$ is diagonal (where $\lambda_1, \ldots , \lambda_{2n} \in \mathbb{C}$ are the eigenvalues of $A$) \cite{Grone87}.
Now partition $Q$ and $D$ as $Q = [ \, Q_1 \quad Q_2 \, ]$ with $Q_1, Q_2 \in \textnormal{M}_{2n \times n}(\CC)$ and $D = D_1 \oplus D_2$ with  $D_1 = \textnormal{diag}(\lambda_1, \ldots , \lambda_n), D_2 = \textnormal{diag}(\lambda_{n+1}, \ldots , \lambda_{2n}) \in \textnormal{M}_{n \times n}(\CC)$. We now obtain from $Q^HAQ = D$ that
\begin{equation} A = \begin{bmatrix} Q_1 & Q_2 \end{bmatrix} \begin{bmatrix} D_1 & \\ & D_2 \end{bmatrix} \begin{bmatrix} Q_1^H \\ Q_2^H \end{bmatrix} = Q_1D_1Q_1^H + Q_2D_2Q_2^H =: E + F \label{equ:BasicDecomp} \end{equation}
holds, where $E = Q_1D_1Q_1^H, F = Q_2D_2Q_2^H \in \textnormal{M}_{2n \times 2n}(\CC)$. Notice that $E$ and $F$ are normal for themselves. Moreover, since $Q$ is unitary, i.e. $Q^HQ = QQ^H = I_n$, we have $Q_1^HQ_2 = Q_2^HQ_1 = 0$. It is now seen directly that $EF = FE = 0$ holds. Beside this property there are no more obvious relations between $E$ and $F$. This situation changes whenever the normal matrix $A$ is (skew)-Hamiltonian or per(skew)-Hermitian. In case of symplectic or perplectic diagonalizability, the matrices $E$ and $F$ are related in a particular way. This relation between $E$ and $F$ is investigated in this section giving some new insights on the symplectic and perplectic diagonalization of those matrices. To this end, the following subsection provides some facts about Lagrangian and neutral subspaces which will be of advantage for our discussion in the sequel. This section is based on \cite[Chap.\,10]{PS2019}.

\subsection{Lagrangian Subspaces}
\label{sec:Lagrangian}
Let $[x,y]=x^HBy$ be either the perplectic form with $B=R_{2n}$ or the symplectic form with $B=J_{2n}$ on $\CC^{2n} \times \CC^{2n}$. In this section we briefly collect some information about neutral subspaces\footnote{The results from this section (in particular Corollary \ref{cor:contained}) are likely to be known \\ although they are not readily found in the literature. They have already been stated in \cite[Sec.\,10.1]{PS2019}.} with respect to the indefinite inner product $[x,y]=x^HBy$.
At first, it is obvious that the set of all neutral subspaces in $\CC^{2n}$ constitutes a partial order under the relation of set-inclusion. That is, for any neutral subspaces $F,G,H \subseteq \CC^{2n}$ we have reflexivity ($F \subseteq F$), transitivity ($F \subseteq G, G \subseteq H$ yields $F \subseteq H$) and anti-symmetry ($F \subseteq G, G \subseteq F$ yields $F=G$). Moreover, for any chain of neutral subspaces $F_1 \subseteq F_2 \subseteq \cdots \subseteq F_k$ the space $F_k$ contains all other spaces from this chain \cite[Def.\,O-1.6]{Gierz03}. In other words, each chain of subspaces has an neutral subspace as an upper bound. According to the lemma of Zorn \cite{Zorn35}, these facts lead to the observation that the (partially ordered) set of neutral subspaces has maximal elements.  The next proposition presents an upper bound for the dimensions of neutral subspaces.

\begin{proposition} \label{prop:maxdimisotrop} \index{dimension!neutral subspace} \index{subspace!neutral!maximal}
For the symplectic inner product $[x,y]=x^HJ_{2n}y$ and the perplectic inner product $[x,y]=x^HR_{2n}y$ on $\CC^{2n} \times \CC^{2n}$ the maximal possible dimension of an neutral subspace is $n$.
\end{proposition}

\begin{proof}
For the Hermitian form $[x,y]=x^HR_{2n}y$ on $\CC^{2n} \times \CC^{2n}$ the statement is proven in \cite[Thm.\,2.3.4]{GoLaRod05} noting that $R_{2n}$ has only the eigenvalues $+1$ and $-1$ with multiplicity $n$. The statement for the symplectic form follows from the same theorem taking into account that the skew-Hermitian form $[x,y]=x^HJ_{2n}y$ and the Hermitian form $[x,y]=x^H(iJ_{2n})y$ have the same neutral subspaces and $iJ_{2n}$ has eigenvalues $+1$ and $-1$ with multiplicities $n$.
\end{proof}

Notice that $\textnormal{im}(S_1)$ and $\textnormal{im}(S_2)$ for any symplectic matrix $[ \, S_1 \quad S_2 \, ] \in \MM_{2n \times 2n}(\CC)$, $S_j \in \MM_{2n \times n}(\CC)$, are neutral of dimension $n$, i.e. Lagrangian (the same holds analogously for perplectic matrices). Thus, the bound given in Proposition \ref{prop:maxdimisotrop} is in both cases sharp. Now it is clear that $\textnormal{im}(S_1)$ has to be a maximal neutral subspace. The following proposition makes a statement on the dimensions of all other maximal neutral subspaces.

\begin{proposition} \label{prop:dimensionmaxsubspace}
For the symplectic inner product $[x,y]=x^HJ_{2n}y$ and the perplectic inner product $[x,y]=x^HR_{2n}y$ on $\CC^{2n} \times \CC^{2n}$ all maximal neutral subspaces have the same dimension. In particular, an neutral subspace is maximal if and only if it is Lagrangian.
\end{proposition}

\begin{proof}
The statement for the Hermitian form $[x,y]=x^HR_{2n}y$ is proven in \cite[\S\,4.2]{Bourbaki07}. The statement on the symplectic form $[x,y]=x^HJ_{2n}y$ follows again from the fact that the Hermitian form $[x,y]=x^H(iJ_{2n})y$ has the same neutral subspaces as $[x,y]=x^HJ_{2n}y$.
\end{proof}

The statement of the following corollary will be important in the upcoming sections.

\begin{corollary} \label{cor:contained}
For the symplectic inner product $[x,y]=x^HJ_{2n}y$ and the perplectic inner product $[x,y]=x^HR_{2n}y$ on $\CC^{2n} \times \CC^{2n}$, each neutral subspace of $\CC^{2n}$ is contained in a Lagrangian subspace.
\end{corollary}

\begin{proof}
Let $F \subseteq \CC^{2n}$ be any neutral subspace. Then $\dim(F) \leq n$ holds according to Proposition \ref{prop:maxdimisotrop}. As the set of all neutral subspaces of $\CC^{2n}$ is partially ordered and has maximal elements, there is always a maximal neutral subspace $G \subseteq \CC^{2n}$ that contains $S$. As all maximal neutral subspaces are Lagrangian according to Proposition \ref{prop:dimensionmaxsubspace}, the statement follows.
\end{proof}

\subsection{Normal (skew)-Hamiltonian Matrices and Symplectic \\ Diagonalizability}
\label{sec:norm_symp}

In this section we consider normal (skew)-Hamiltonian matrices and analyze their properties with respect to (simultaneous) symplectic and unitary diagonalization. A key fact used in the subsequent analysis is that matrices which are unitary and symplectic (for which we use the abbreviation \textit{unitary-symplectic}) have a very special form, cf. Proposition \ref{prop:unitarysymplectic} (see also \cite{PaVL81}). Theorem \ref{thm:unitsympdiag} shows that unitary and symplectic diagonalizations of any normal (skew)-Hamiltonian matrix are always compatible and simultaneously achievable. This is the basic insight underlying the decompositions presented in Theorem \ref{thm:master1}.

\begin{proposition} \label{prop:unitarysymplectic}
A matrix $Q \in \MM_{2n \times 2n}(\CC)$ is unitary-symplectic if and only if $Q = [ \, V \quad J_{2n}^TV \, ]$ for some matrix $V \in \textnormal{M}_{2n \times n}(\CC)$ with $V^HV = I_n$ and $V^HJ_{2n}V = 0$.
\end{proposition}

\begin{proof}
Let $Q  = [ \, Q_1 \quad Q_2 \, ] $ be unitary-symplectic with $Q_1, Q_2 \in \textnormal{M}_{2n \times n}(\CC).$ As $Q$ is unitary we have $Q^HQ=I_{2n}$ and as it is symplectic $Q^HJ_{2n}Q = J_{2n}$ holds. Multiplying the latter with $Q$ from the left gives $J_{2n}Q = QJ_{2n}$, so $Q$ commutes with $J_{2n}$. From this relation it follows that $J_{2n}Q_1 = -Q_2$, i.e. $Q_2 = J_{2n}^TQ_1.$ Moreover, from $Q^HJ_{2n}Q = J_{2n}$ it follows that $Q_1^HJ_{2n}Q_1 = 0$. Now let $Q = [ \, V \quad J_{2n}^TV \, ]$ with $V^HV = I_n$ and $V^HJ_{2n}V = 0$ be given. We have
\begin{align*} \begin{bmatrix} V & J_{2n}^TV \end{bmatrix}^HJ_{2n} \begin{bmatrix} V & J_{2n}^TV \end{bmatrix} &= \begin{bmatrix} V^H \\ V^HJ_{2n} \end{bmatrix} J_{2n} \begin{bmatrix} V & J_{2n}^T V \end{bmatrix} \\ &= \begin{bmatrix} V^HJ_{2n}V & V^HV \\ - V^HV & - V^HJ_{2n}^TV \end{bmatrix} = \begin{bmatrix} & I_n \\ - I_n & \end{bmatrix}. \end{align*}

which yields $Q^HJ_{2n}Q = J_{2n}$. This completes the proof.
\end{proof}

In other words, Proposition \ref{prop:unitarysymplectic} states that $Q = [ \, V \quad J_{2n}^TV \, ] \in \MM_{2n \times 2n}(\CC)$ is unitary-symplectic if and only if the columns of $V \in \MM_{2n \times n}(\CC)$ are orthonormal and span a Lagrangian subspace. Recall that a diagonal Hamiltonian matrix $\widetilde{D} \in \textnormal{M}_{2n \times 2n}(\CC)$  has the form
\begin{equation} \widetilde{D} = \begin{bmatrix} D & 0 \\0 & - D^H \end{bmatrix} \quad \textnormal{with} \; D = \textnormal{diag}(\, \lambda_1, \ldots , \lambda_n \, ) \in \textnormal{M}_{n \times n}(\CC). \label{equ:diagHamiltonian00} \end{equation}
The following Theorem \ref{thm:unitsympdiag} gives a condition for the existence of a unitary-symplectic diagonalization of a normal (skew)-Hamiltonian matrix. In particular, it turns out that the symplectic diagonalizability is always sufficient. We prove the statement only for Hamiltonian matrices as the proof works analogously in the skew-Hamiltonian case.

\begin{theorem} \label{thm:unitsympdiag}
A normal (skew)-Hamiltonian matrix $A \in \textnormal{M}_{2n \times 2n}(\CC)$ is symplectic diagonalizable if and only if it is unitary-symplectic diagonalizable.
\end{theorem}

\begin{proof}
$ \Rightarrow$ Let $A \in \textnormal{M}_{2n \times 2n}(\CC)$ be normal Hamiltonian and symplectic diagonalizable via $T  \in \textnormal{M}_{2n \times 2n}(\CC)$, i.e.
$$ T^{-1}AT = \begin{bmatrix} D & \\ & -D^H \end{bmatrix}, \quad T^HJ_{2n}T = J_{2n}, \quad D = \textnormal{diag}(\lambda_1, \ldots , \lambda_n).$$
Let $ T= [ \, T_1 \quad T_2 \, ]$ with $T_1 = [ \, t_1 \; \cdots \; t_n \, ], T_2 \in \MM_{2n \times n}(\CC)$. The fact that $T^HJ_{2n}T = J_{2n}$ holds reveals that $\textnormal{span}(t_1, \ldots , t_n)$ is a Lagrangian subspaces (as is $\textnormal{span}(T_2)$). Due to the normality of $A$, eigenspaces for different eigenvalues of $A$ are orthogonal to each other. Whenever any $\lambda_j$ appears $r$ times in $D$ (in positions $j_1, \ldots , j_r$, say), we orthogonalize and normalize the corresponding eigenvectors $t_{j_1}, \ldots , t_{j_r}$ from $T_1$ obtaining $s_{j_1}, \ldots , s_{j_r}$. In particular, whenever $\lambda_k$ appears only once in $D$ (in position $k$), the sole eigenvector $t_k$ is replaced by its normalized version $ s_k = t_k/\Vert t_k \Vert_2$. The $n$ vectors obtained from this orthogonalization procedure are collected in a matrix $S \in \MM_{2n \times n}(\CC)$, that is, $S = [ \, s_1 \; \cdots \; s_n \, ]$, and we set $\widehat{S} = [ \, S \quad J_{2n}^T S \, ] \in \textnormal{M}_{2n \times 2n}(\CC)$. Now $s_1, \ldots , s_n$ are $n$ orthonormal eigenvectors of $A$ with $\textnormal{span}(T_1) = \textnormal{span}(S)$, i.e. $\textnormal{span}(S)$ is still Lagrangian. According to Proposition \ref{prop:unitarysymplectic} $\widehat{S}$ is unitary-symplectic. Moreover,
\begin{equation} \widehat{A} := \widehat{S}^H A \widehat{S} = \begin{bmatrix} S^H \\ S^H J_{2n} \end{bmatrix} A \begin{bmatrix} S & J_{2n}^T S \end{bmatrix} = \begin{bmatrix} S^HAS & S^HAJ_{2n}^TS \\ S^HJ_{2n}AS & S^HJ_{2n}AJ_{2n}S \end{bmatrix}. \label{equ:diag1} \end{equation}
As $AS = SD$ holds (following from $AT_1 = T_1D$ and the construction of $S$), we have $S^HAS = D$ in \eqref{equ:diag1} using the fact that $S^HS = I_n$. Moreover, $S^HJ_{2n}AS = S^HJ_{2n}SD = 0$ holds since $\textnormal{im}(S)$ is a Lagrangian subspace, i.e. $S^HJ_{2n}S = 0$. As $\widehat{S}$ is symplectic, $\widehat{A}$ remains to be Hamiltonian. This implies $S^HJ_{2n}AJ_{2n}S$  in \eqref{equ:diag1} to be equal to $-D^H$. Therefore, we showed that $\widehat{A}$ is actually upper-triangular. However, since $\widehat{S}$ is unitary, the normality of $A$ is preserved in $\widehat{A}$. As a normal upper-triangular matrix must be diagonal, $\widehat{S}^H A \widehat{S}$ is a unitary-symplectic diagonalization of $A$.

$\Leftarrow$ This is clear.
\end{proof}

The next Theorem \ref{thm:master1} states a special property of normal Hamiltonian matrices $A \in \MM_{2n \times 2n}(\CC)$ which are symplectic diagonalizable. In this case, the unitary-symplectic diagonalizability according to Theorem \ref{thm:unitsympdiag} reveals the existence of a specially structured additive decomposition of $A$ similar to the one from \eqref{equ:BasicDecomp}. As will be shown next, this decomposition is actually equivalent to $A$ being symplectic diagonalizable. Theorem \ref{thm:master1} is the main result of this section.

\begin{theorem} \label{thm:master1}
A matrix $A \in \textnormal{M}_{2n \times 2n}(\CC)$ is normal Hamiltonian and symplectic diagonalizable if and only if $A = N - N^\star$ for some normal matrix $N \in \textnormal{M}_{2n \times 2n}(\CC)$ satisfying $NN^\star = N^\star N = 0$.
\end{theorem}

\begin{proof}
$\Rightarrow$ Let $A \in \MM_{2n \times 2n}(\CC)$ be normal Hamiltonian and assume that $A$ is symplectic diagonalizable. According to Theorem \ref{thm:unitsympdiag}, there exists a unitary-syplectic diagonalization $U^HAU = \widetilde{D}$ of $A$ and, by Proposition \ref{prop:unitarysymplectic}, $U = [ \, V \; J_{2n}^HV \, ]$ for some matrix $V \in \MM_{2n \times n}(\CC)$ with $V^HV = I_n$ and $V^HJ_{2n}V = 0$. Moreover, $\widetilde{D}$ has the form given in \eqref{equ:diagHamiltonian} for some matrix $D = \textnormal{diag}(\, \lambda_1, \ldots , \lambda_n \, ) \in \MM_{n \times n}(\CC)$. Then
\begin{align}
A = U\widetilde{D}U^H &= \begin{bmatrix} V & J_{2n}^HV \end{bmatrix}  \begin{bmatrix} D & 0 \\ 0 & - D^H \end{bmatrix}  \begin{bmatrix} V^H \\ V^HJ_{2n} \end{bmatrix} \label{equ:HamilProof0} \\
          &= V D V^H - J_{2n}^HV D^H V^HJ_{2n} = N - N^\star \notag
\end{align}
for $N := V D V^H$.
Moreover, $N$ is normal as $NN^H = V D V^H V D^H V^H = V D D^H V^H$ coincides with $N^HN = V D^H D V^H$ since $D$ is diagonal. Furthermore, we have
$$ NN^\star = V D V^H \big( J_{2n}^H V D^H V^H  J_{2n} \big) = - V D \big( V^H  J_{2n} V \big) D^H V^H J_{2n} = 0 \label{equ:HamilProof1} $$
as $V^H J_{2n}V = 0$.
Similarly it can be seen that $N^\star N = 0$ holds.

$\Leftarrow$ Now assume that $A = N-N^\star$ holds for some normal matrix $N \in \MM_{2n \times 2n}(\CC)$ with $NN^\star = N^\star N = 0_{2n \times 2n}$. Then $A$ is Hamiltonian since $A^\star = (N - N^\star)^\star = N^\star - N = - (N - N^\star) = -A$.
Furthermore, we have
\begin{equation} \begin{aligned}
A^HA &= (N - N^\star)^H(N - N^\star) = N^HN - (N^\star)^HN - N^HN^\star + (N^\star)^HN^\star  \\
AA^H &=  (N - N^\star)(N - N^\star)^H = NN^H - N^\star N^H - N(N^\star)^H + N^\star ( N^\star)^H.
\end{aligned} \label{equ:normalA} \end{equation}
Recall that the normality of $N^\star$ follows directly from the normality of $N$. With this observation, the assumption $NN^\star = N^\star N$ and the normality of $N^\star$ imply $N(N^\star)^H = (N^\star)^HN$ according to \cite[Sec.~2(6)]{Grone87}. Similarly, we obtain $N^\star N^H = N^H N^\star$ and both expressions in \eqref{equ:normalA} coincide. Thus, $A$ is normal. \\
\indent Moreover, as $N^\star N = J_{2n}^TN^HJ_{2n}N = 0$, multiplication from the left with $J_{2n}$ yields $N^HJ_{2n}N = 0$, so the columns of $N$ span an neutral subspace. Furthermore, $N^\star N=0$ implies $\textnormal{im}(N) \subseteq \textnormal{null}(N^\star)$ which yields $\textnormal{rank}(N) \leq n$ since\footnote{Alternatively, $\textnormal{rank}(N) \leq n$ follows from the isotropy of $\textnormal{im}(N)$ using the result from Proposition \ref{prop:maxdimisotrop}.}
$$\textnormal{rank}(N) = \dim(\textnormal{im}(N)) \leq \dim(\textnormal{null}(N^\star)) = 2n - \textnormal{rank}(N^\star) = 2n - \textnormal{rank}(N).$$
Now, the normality of $N$ and $\textnormal{rank}(N) \leq n$ imply that there exists a diagonal matrix $D \in \MM_{n \times n}(\CC)$, $\textnormal{rank}(D) = \textnormal{rank}(N)$, and a matrix $V \in \MM_{2n \times n}(\CC)$ with orthonormal columns (i.e. $V^HV = I_n$) so that $N = V DV^H$. If $\textnormal{rank}(N)=k < n$, then $D$ has $n-k$ eigenvalues equal to zero. Without loss of generality, we assume that these zeros appear in the trailing $n-k$ diagonal positions in $D$. The expression of $N$ implies $N^\star = J_{2n}^TN^HJ_{2n} = J_{2n}^TV D^H V^H J_{2n}$. Therefore, $A$ can be expressed as
\begin{equation}  A = N - N^\star = V DV^H - J_{2n}^TV D^H V^H J_{2n}. \label{equ:V} \end{equation}
With $\widetilde{D} := \left[ \begin{smallmatrix} D & \\ &   -D^H  \end{smallmatrix} \right] $ and $U := [ \, V \quad J_{2n}^T V \, ]$ we observe in accordance with \eqref{equ:V} that
\begin{equation} U\widetilde{D}U^H = \begin{bmatrix}  V & J_{2n}^TV  \end{bmatrix} \begin{bmatrix} D & \\ & -D^H \end{bmatrix} \begin{bmatrix} V^H \\  V^H J_{2n} \end{bmatrix} = V DV^H - J_{2n}^TV D^H V^H J_{2n} = A. \label{equ:decofA} \end{equation}
Then, obviously, $U^HAU = \widetilde{D}$ is diagonal. Unfortunately, as long as $V^HJ_{2n}V = 0$ does not holds, $U$ will neither be unitary nor symplectic. However, if it can be shown that $\textnormal{im}(V)$ is in fact a Lagrangian subspace, Proposition \ref{prop:unitarysymplectic} applies and the theorem is proven. We distinguish between the two cases $\textnormal{rank}(N) = n$ and $\textnormal{rank}(N) = k < n$. \\
\indent First assume that $\textnormal{rank}(N)=n$, i.e. $\dim( \textnormal{im}(N))=n$. Then we have $\textnormal{rank}(D)=n$ and therefore $\textnormal{im}(N) = \textnormal{im}(V)$ is a Lagrangian subspace. As $V^HV = I_n$ holds, Proposition \ref{prop:unitarysymplectic} yields that $U$ is unitary-symplectic and $U^HAU=\widetilde{D}$ is a unitary-symplectic diagonalization of $A$.
Now let $\textnormal{rank}(N) =k < n$. Recall that we assumed the $n-k$ eigenvalues of $D$ which are equal to zero to appear in its trailing $n-k$ diagonal positions. Then, if $V = [ \, v_1 \; \cdots \; v_n \, ]$ it is immediate that $\textnormal{im}(N)$ coincides with the $\textnormal{span}(v_1, \ldots , v_k)$. In other words, the last $n-k$ columns $v_{k+1}, \ldots , v_n$ of $V$ have no contribution to the matrices $N$, $N^\star$ or $A$ at all. Therefore, as long as the orthogonality constraint is met, $v_{k+1}, \ldots , v_n$ can be replaced by any other columns without changing the expression of $A$ in \eqref{equ:decofA}.
Now we take Corollary \ref{cor:contained} into account. As $\textnormal{span}(v_1, \ldots , v_k) = \textnormal{im}(N)$ is an neutral subspace (of dimension $k$), it is properly contained in a Lagrangian subspace. Therefore, there exist $n-k$ vectors $\widetilde{v}_{k+1}, \ldots , \widetilde{v}_n \in \CC^{2n}$ such that $\textnormal{span}(v_1, \ldots , v_k, \widetilde{v}_{k+1}, \ldots , \widetilde{v}_n)$ is a Lagrangian subspace. If $\widetilde{v}_{k+1}, \ldots , \widetilde{v}_n$ are chosen so that $$\widetilde{V} := \big[ \, v_1 \; \cdots \; v_k \; \widetilde{v}_{k+1} \; \cdots \; \widetilde{v}_n \, \big] \in \MM_{2n \times n}(\CC)$$ has orthonormal columns, i.e. $\widetilde{V}^H \widetilde{V} = I_n$, we obtain
\begin{equation} \begin{bmatrix} \widetilde{V} & J_{2n}^T\widetilde{V}  \end{bmatrix} \begin{bmatrix} D & \\ & -D^H \end{bmatrix} \begin{bmatrix} \widetilde{V}^H \\ \widetilde{V}^H J_{2n} \end{bmatrix} = \widetilde{V} D\widetilde{V}^H - J_{2n}^T \widetilde{V} D^H \widetilde{V}^H J_{2n} = A. \label{equ:decofA1} \end{equation}
Now the matrix $\widetilde{U} := [ \, \widetilde{V} \quad J_{2n}^T\widetilde{V} \, ] \in \MM_{2n \times 2n}(\CC)$ is unitary-symplectic according to Proposition \ref{prop:unitarysymplectic} and $\widetilde{U}^HA\widetilde{U} = \widetilde{D}$ is a unitary-symplectic diagonalization of $A$.
\end{proof}

\begin{remark}
Consider once again the proof of Theorem \ref{thm:master1} and a decomposition $A = N - N^\star$ for some normal Hamiltonian matrix $A \in \textnormal{M}_{2n \times 2n}(\CC)$ with $NN^\star = N^\star N = 0$ and normal $N \in \textnormal{M}_{2n \times 2n}(\CC)$.
\begin{enumerate}
\item For the matrix $N \in \textnormal{M}_{2n \times 2n}(\CC)$ it always holds that $\textnormal{rank}(N) = \textnormal{rank}(N^\star)  \leq n$ and $\textnormal{im}(N)$, $\textnormal{im}(N^\star)$ are neutral subspaces. Moreover, as $AN = (N-N^\star)N = N^2$ and $AN^\star = (N^\star)^2$, both $\textnormal{im}(N)$ and $\textnormal{im}(N^\star)$ are invariant for $A$. In conclusion, $\textnormal{im}(N)$ and $\textnormal{im}(N^\star)$ are invariant Lagrangian subspaces for $A$ if $\textnormal{rank}(N) = \textnormal{rank}(N^\star)=n$.
\item If $(\lambda,v)$ is an eigenpair of $N$, i.e. $Nv = \lambda v$, and $\lambda \neq 0$, then $$Av = \frac{1}{\lambda}(N-N^\star)(\lambda v) = \frac{1}{\lambda} (N-N^\star)Nv = \frac{1}{\lambda} N^2 v = \lambda v,$$
so $\lambda$ is an eigenvalue of $A$ with eigenvector $v$. In particular, we have $\sigma(N) \setminus \lbrace 0 \rbrace \subseteq \sigma(A)$. Similarly it can be shown that $\sigma(N^\star) \setminus \lbrace 0 \rbrace \subseteq \sigma(A)$. In conclusion, whenever $\textnormal{rank}(N) = \textnormal{rank}(N^\star)=n$, the matrix $A$ is nonsingular (i.e. $0 \notin \sigma(A)$) and it holds that $(\sigma(N) \cup \sigma(N^\star)) \setminus \lbrace 0 \rbrace = \sigma(A)$.
\end{enumerate}
\end{remark}

The additive decomposition $A = N - N^\star \in \MM_{2n \times 2n}(\CC)$ (for $N$ being normal with $NN^\star = N^\star N = 0$) proven in Theorem \ref{thm:master1} can be used to easily derive some nice consequences whenever not $A$ itself but some \textit{expression in} $A$ is considered. One such situation is given by considering the exponential of $A$ \cite[Sec.\,10]{High08}. Recall that the exponential of a Hamiltonian matrix yields a symplectic matrix \cite[Sec.\,7.2]{Still08}.

\begin{example} \label{ex:exp1}
Let $A =N - N^\star \in \textnormal{M}_{2n \times 2n}(\CC)$ be normal Hamiltonian with some normal $N \in \textnormal{M}_{2n \times 2n}(\CC)$ satisfying $NN^\star = N^\star N = 0$. Considering the exponential $\exp(A)$ of $A$ we obtain
\begin{align*} \exp(A) &= \exp(N - N^\star) = \exp(N)\exp(-N^\star) = \exp(N) \exp(N^\star)^{-1} \\ &= \exp(N) \exp(J_{2n}^TN^HJ_{2n})^{-1} = \exp(N) \big( J_{2n}^T \exp(N^H) J_{2n} \big)^{-1} \\ &= \exp(N) \big( J_{2n}^T \exp(N)^H J_{2n} \big)^{-1} = \exp(N) \big( \exp(N)^\star \big)^{-1}
\end{align*}
where we have used the facts that $\exp(-N^\star) = \exp(N^\star)^{-1}$ and $$\exp(J_{2n}^{-1}N^HJ_{2n}) = J_{2n}^{-1} \exp(N)^H J_{2n},$$ cf. \cite{High08}.
Notice that the exponential of a normal matrix remains to be normal. Therefore, the symplectic and normal matrix $\exp(A)$ can be decomposed as $S(S^\star)^{-1} = SS^{- \star}$ for some normal matrix $S \in \MM_{2n \times 2n}(\CC)$. If $A = N+N^\star$ is skew-Hamiltonian with $N$ normal and $NN^\star = N^\star N = 0$, the same derivation shows that $\exp(A) = SS^\star$ (for $S = \exp(N)$) revealing nicely the maintained skew-Hamiltonian structure. Certainly, $\exp(A)$ is again normal.
\end{example}

Theorem \ref{thm:master1} directly extends to normal skew-Hamiltonian matrices which are unitary-symplectic diagonalizable. To this end, notice that a diagonal skew-Hamiltonian matrix $\widetilde{D} \in \textnormal{M}_{2n \times 2n}(\CC)$ has the form given in \eqref{equ:diagHamiltonian}.
Thus, the only significant difference comparing the proofs of Theorem \ref{thm:master2} and Theorem \ref{thm:master1} above is a change of sign. Consequently, the proof of  Theorem \ref{thm:master2} is omitted.

\begin{theorem} \label{thm:master2}
A matrix $A \in \textnormal{M}_{2n \times 2n}(\CC)$ is normal skew-Hamiltonian and symplectic diagonalizable if and only if $A = N + N^\star$ for some normal matrix $N \in \textnormal{M}_{2n \times 2n}(\CC)$ satisfying $NN^\star = N^\star N = 0$.
\end{theorem}

As the next example shows, the special additive decomposability of a normal skew-Hamiltonian matrix $A = N + N^\star \in \MM_{2n \times 2n}(\CC)$ (with $N$ normal and $NN^\star = N^\star N = 0$) carries over to other matrix functions as, e.g., matrix roots. In particular, a matrix root of $A$ can be expressed by an analogous decomposition as $A$ replacing $N$ by its matrix root.

\begin{example} \label{ex:roots}
Let $A =N + N^\star \in \textnormal{M}_{2n \times 2n}(\CC)$ be nonsingular, normal and skew-Hamiltonian with $N = V DV^H$ as in \eqref{equ:HamilProof0} and $D = \textnormal{diag}(\lambda_1, \ldots , \lambda_n)$ such that
\begin{align*}
A = U\widetilde{D}U^H &= \begin{bmatrix} V & J_{2n}^HV \end{bmatrix}  \begin{bmatrix} D & 0 \\ 0 & D^H \end{bmatrix}  \begin{bmatrix} V^H \\ V^HJ_{2n} \end{bmatrix} \label{equ:HamilProof0} \\
          &= V D V^H + J_{2n}^HV D^H V^HJ_{2n} = N + N^\star. \notag
\end{align*}
Define $D^{1/2} = \textnormal{diag}(\lambda_1^{1/2}, \ldots , \lambda_n^{1/2})$ and $N^{1/2} := V D^{1/2}V^H$ (where $^{1/2}$ denotes any square root). Then $N^{1/2}$ is a square root of $N$, that is, $(N^{1/2})^2 = N$. Moreover, $((N^{1/2})^\star)^2 = N^\star$ can be verified by a direct calculation and it still holds that $N^{1/2} (N^{1/2})^\star = (N^{1/2})^\star N^{1/2} = 0$ due to the construction of $N^{1/2}$.  Therefore we obtain
\begin{align*}  \big( N^{1/2} + (N^{1/2})^\star \big) \big( N^{1/2} + (N^{1/2})^\star \big) &= (N^{1/2})^2 + N^{1/2}(N^{1/2})^\star + (N^{1/2})^\star N^{1/2} + \big( (N^{1/2})^\star \big)^2 \\ &= N + N^\star = A \end{align*}
    and $N^{1/2} + (N^{1/2})^\star$ is a normal skew-Hamiltonian square root of $A$ which is, by Theorem \ref{thm:master2}, again symplectic diagonalizable.  Certainly, this result can be generalized to arbitrary matrix $p$th roots for any $p \in \mathbb{N}$.
\end{example}

\subsection{Normal per(skew)-Hermitian Matrices and Perplectic \\ Diagonalizability}
\label{sec:norm_perp}

Now we turn our attention to normal matrices  which are per-Hermitian or perskew-Hermitian and analyze their properties with respect to unitary and perplectic diagonalization. The main statements are similar to the previous results from Section \ref{sec:norm_symp} although the indefinite inner product $[x,y]=x^HR_{2n}y$ on $\CC^{2n} \times \CC^{2n}$ under consideration is now Hermitian instead of skew-Hermitian. We begin with the characterization of matrices which are both unitary and perplectic in Proposition \ref{prop:unitaryperplectic} (we use the abbreviation \textit{unitary-perplectic} for these matrices). The statement analogous to Theorem \ref{thm:unitsympdiag} on unitary-perplectic diagonalizability is presented in Theorem \ref{thm:unitperpdiag} whereas the analogous results to Theorem \ref{thm:master1} and \ref{thm:master2} are given in Theorem \ref{thm:master3}.

\begin{proposition} \label{prop:unitaryperplectic}
A matrix $Q \in \MM_{2n \times 2n}(\CC)$ is unitary-perplectic if and only if $Q = [ \, V \; R_{2n}VR_n \, ]$ for some matrix $V \in \textnormal{M}_{2n \times n}(\CC)$ with $V^HV = I_n$ and $V^HR_{2n}V = 0$.
\end{proposition}

\begin{proof}
Let $Q  = [ \, Q_1 \quad Q_2 \, ]$ be unitary-perplectic with $Q_1, Q_2 \in \textnormal{M}_{2n \times n}(\CC).$ As $Q$ is unitary we have $Q^HQ=I_{2n}$ and as it is perplectic $Q^HR_{2n}Q = R_{2n}$ holds. Multiplying the latter with $Q$ from the left gives $R_{2n}Q = QR_{2n}$, so $Q$ commutes with $R_{2n}$. Matrices satisfying this condition are known as centrosymmetric \cite[Def.\,2.2]{Abu02}. %\cite{Reid97}.
It is easy to see that any centrosymmetric matrix $C \in \MM_{2n \times 2n}(\CC)$ is symmetric with respect to the center of it and thus can be expressed as $C = [ \, W \quad R_{2n}WR_n \, ]$ for some $W \in \textnormal{M}_{2n \times n}(\CC).$ Moreover, any matrix of the form of $C$ is centrosymmetric for any $W$. Now let $Q = [ \, V \quad R_{2n}VR_n \, ]$ with $V^HV = I_n$ and $V^HR_{2n}V = 0$ be given. Then we have
$$ \begin{aligned} \begin{bmatrix} V & R_{2n}VR_n \end{bmatrix}^HR_{2n} \begin{bmatrix} V & R_{2n}^TVR_n \end{bmatrix} &= \begin{bmatrix} V^H \\ R_nV^HR_{2n} \end{bmatrix} R_{2n} \begin{bmatrix} V & R_{2n} VR_n \end{bmatrix} \\ &= \begin{bmatrix} V^HR_{2n}V & V^HVR_n \\ R_nV^HV & R_n V^HR_{2n}VR_n \end{bmatrix} = \begin{bmatrix} & R_n \\ R_n & \end{bmatrix}. \end{aligned}$$
which gives $Q^HR_{2n}Q = R_{2n}$. This completes the proof.
\end{proof}

The analogous result to Theorem \ref{thm:unitsympdiag} is stated in the following proposition.

\begin{theorem} \label{thm:unitperpdiag}
A normal per(skew)-Hermitian matrix $A \in \textnormal{M}_{2n \times 2n}(\CC)$ is perplectic diagonalizable if and only if it is unitary-perplectic diagonalizable.
\end{theorem}

The proof of Theorem \ref{thm:unitperpdiag} goes along the same lines as that of Theorem \ref{thm:unitsympdiag} noting that, for any perplectic matrix $P =[ \, P_1 \quad P_2 \, ] \in \MM_{2n \times 2n}(\CC)$ with $P_1,P_2 \in \MM_{2n \times n}(\CC)$, $\textnormal{span}(P_1)$ and $\textnormal{span}(P_2)$ are Lagrangian subspaces. The same orthogonalization procedure of the eigenvectors of $A$ given by the columns of $P_1$ as discussed in the proof of Theorem \ref{thm:unitsympdiag} then admits the construction of a unitary-perplectic matrix (characterized by Proposition \ref{prop:unitaryperplectic}) which diagonalizes $A$.

Notice that a diagonal per(skew)-Hermitian matrix $\widetilde{D} \in \textnormal{M}_{2n \times 2n}(\CC)$ has the form
\begin{equation} \widetilde{D} = \begin{bmatrix} D & \\ & \pm R_n D^H R_n \end{bmatrix} \quad \textnormal{with} \; D = \textnormal{diag}(\lambda_1, \ldots , \lambda_n) \in \MM_{n \times n}(\CC). \label{equ:perherm_diag} \end{equation}
The characterization of unitary-perplectic matrices in Proposition \ref{prop:unitaryperplectic} together with \eqref{equ:perherm_diag} admit a proof analogous to that of Theorem \ref{thm:master1} for the following results.

\begin{theorem}  \label{thm:master3}
\begin{enumerate}
\item A matrix $A \in \textnormal{M}_{2n \times 2n}(\CC)$ is normal per-Hermitian and perplectic diagonalizable if and only if $A = N + N^\star$ for some normal matrix $N \in \textnormal{M}_{2n \times 2n}(\CC)$ with $NN^\star = N^\star N = 0$.
\item A matrix $A \in \textnormal{M}_{2n \times 2n}(\CC)$ is normal perskew-Hermitian and perplectic diagonalizable if and only if $A = N - N^\star$ for some normal matrix $N \in \textnormal{M}_{2n \times 2n}(\CC)$ satisfying $NN^\star = N^\star N = 0$.
\end{enumerate}
\end{theorem}

Comparing Theorem \ref{thm:master1} and Theorem \ref{thm:master2} to Theorem \ref{thm:master3} notice that the decomposition $A = N \pm N^\star$ always carries a `--' sign whenever $A$ is skewadjoint and a `+' sign if $A$ is selfadjoint with respect to the indefinite inner products $[x,y]=x^HJ_{2n}y$ and $[x,y]=x^HR_{2n}y$, respectively.
It can be shown analogously to Example \ref{ex:exp1} that the exponential $\exp(A)$ of any normal per(skew)-Hermitian matrix $A  = N \pm N^\star$ (with normal $N \in \MM_{2n \times 2n}(\CC)$ satisfying $NN^\star = N^\star N = 0$) can be expressed as $\exp(A) = PP^{\pm \star}$ for the normal matrix $P = \exp(N)$. In particular, whenever $A$ is normal perskew-Hermitian, then $\exp(A)$ is normal and perplectic with an expression of the form $\exp(A) = PP^{- \star}$ for a normal matrix $P$. Similarly, the result from Example \ref{ex:roots} extends by the same reasoning to per-Hermitian matrices.

\section{Conclusions}
\label{sec:conclusions}
In this work we analyzed (skew)-Hamiltonian and per(skew)-Hermitian matrices under the viewpoint of structure-preserving diagonalizability. We showed that the symplectic and perplectic diagonalization of such matrices  is possible if and only if certain conditions apply to their real or purely imaginary eigenvalues and corresponding eigenspaces (cf. Theorems \ref{thm:diag_sympl0} and \ref{thm:diag_perpl0}). This diagonalizability condition turned out to be essentially the same for (skew)-Hamiltonian and per(skew)-Hermitian matrices although their structures are determined by a skew-Hermitian indefinite inner product and a Hermitian indefinite inner product, respectively. We conferred special attention to those structured matrices which are additionally normal. In this case, it was shown that an existing symplectic or perplectic diagonalization is a sufficient criterion to guarantee a diagonalization by a unitary-symplectic or unitary-perplectic similarity transformation to exist (Theorem \ref{thm:unitperpdiag} and \ref{thm:unitsympdiag}). For normal (skew)-Hamiltonian and per(skew)-Hermitian matrices it was proven that a symplectic or perplectic transformation to diagonal form implies the existence of a structured additive decomposition of such matrices. In turn, such an additive decomposition was shown to imply the matrix at hand to be unitary-symplectic or unitary-perplectic diagonalizable and gave an alternative characterization of such matrices (Theorems \ref{thm:master1}, \ref{thm:master2} and \ref{thm:master3}). The proof of this fact essentially required the knowledge that every neutral subspace is contained in a maximal neutral subspace (the latter has been called Lagrangian subspace, cf. Corollary \ref{cor:contained}). Throughout this work, some examples have been provided to illustrate the obtained results.

% Bibliography
%-----------------------------------------------------------------

\end{document}